\documentclass[12pt]{article} 

\usepackage{amsmath}
\usepackage{amsthm}
\usepackage{amsfonts}
\usepackage{mathrsfs}
\usepackage{stmaryrd}
\usepackage{setspace}
\usepackage{fullpage}
\usepackage{amssymb}
\usepackage{breqn}
\usepackage{enumitem}
\usepackage{bbold} 
\usepackage{authblk}
\usepackage{comment}
\usepackage{hyperref}
\usepackage{pgf,tikz}
\usepackage{graphicx}
\usepackage{subcaption}

\newtheorem{thm}{Theorem}[section]

\newtheorem{prop}[thm]{Proposition}

\newtheorem{corollary}[thm]{Corollary}

\newtheorem{clm}[thm]{Claim}

\newcommand\ex{\ensuremath{\mathrm{ex}}}

\newcommand\cF{{\mathcal F}}

\newtheorem*{thm*}{Theorem}
\newtheorem*{prop*}{Proposition}
\newcommand{\ignore}[1]{}

\title{On Tur\'an problems with bounded matching number}
\author{Dániel Gerbner} 

\date{\small Alfr\'ed R\'enyi Institute of Mathematics\\
\small \texttt{gerbner.daniel@renyi.hu}}

\begin{document}

\maketitle

\begin{abstract}
  Very recently, Alon and Frankl initiated the study of the maximum number of edges in $n$-vertex $F$-free graphs with matching number at most $s$. For fixed $F$ and $s$, we determine this number apart from a constant additive term. We also obtain several exact results.
\end{abstract}

\textbf{Keywords.}    Tur\'an number; matching

\section{Introduction}

A basic problem in extremal graph theory is the following. Given a positive integer $n$ and a graph $F$, how many edges can an $n$-vertex graph have if it does not contain $F$ as a subgraph? More generally, given $n$ and a family $\cF$ of graphs, how many edges can an $n$-vertex graph have if it does not contain any member of $\cF$ as a subgraph? We denote the largest number of edges by
$\ex(n,\cF)$. In the case $\cF$ contains only one graph, we write $\ex(n,F)$ instead of $\ex(n,\{F\})$.

One of the earliest results concerning these numbers is due to Tur\'an \cite{T}, who showed that $\ex(n,K_{k+1})=|E(T(n,k))|$, where the Tur\'an graph $T(n,k)$ is the complete $k$-partite $n$-vertex graph with each part of order $\lfloor n/k\rfloor$ or $\lceil n/k\rceil$. Another fundamental result is due to Erd\H os and Gallai \cite{eg}, who showed that $\ex(n,M_{s+1})=\max\{|E(G(n,s))|,\binom{2s+1}{2}\}$, where the matching $M_{s+1}$ consists of $s+1$ independent edges and $G(n,s)$ has $s$ vertices of degree $n-1$ and $n-s$ vertices of degree $s$. Chv{\'a}tal and Hanson \cite{ch} determined $\ex(n,K_{1,k+1},M_{s+1})$ (the case $s=k$ was solved earlier in \cite{ahs}).

Very recently, Alon and Frankl \cite{af} combined the above results and considered forbidding a graph $F$ and $M_{s+1}$ at the same time. 
Let $G(n,k,s)$ denote the complete $k$-partite $n$-vertex graph with one part of order $n-s$ and each other part of order $\lfloor s/k\rfloor$ or $\lceil s/k\rceil$. Alon and Frankl \cite{af} showed that $\ex(n,\{K_{k+1},M_{s+1}\})=\max\{|E(G(n,k,s))|,|E(T(2s+1,k))|\}$, in particular for $n$ sufficiently large we have $\ex(n,\{K_{k+1},M_{s+1}\})=|E(G(n,k,s))|$. Moreover, for any $F$ with chromatic number $k+1$ and a color-critical edge (an edge whose deletion decreases the chromatic number), they showed that $\ex(n,\{F,M_{s+1}\})=|E(G(n,k,s))|$, provided $s>s_0(F)$ and $n>n_0(F)$.


First we prove a generalization of this second result.

\begin{thm}\label{thm1}
If $\chi(F)>2$ and $n$ is large enough, then $\ex(n,\{F,M_{s+1}\})=\ex(s,\cF)+s(n-s)$, where $\cF$ is the family of graphs obtained by deleting an independent set from $F$.
\end{thm}

We remark that isolated vertices of members of $\cF$ are important here. For example, if $F$ is an odd cycle $C_{2\ell+1}$ (or more generally, if $F$ is 3-chromatic with a color-critical edge), then $\cF$ contains the graph consisting of an edge and $\ell-1$ isolated vertices. If $s\ge \ell+1$, then $\ex(s,\cF)=0$, while if $s<\ell+1$, then $\ex(s,\cF)=\binom{s}{2}$.

Observe that if $F$ has a color-critical edge, then $\cF$ contains a graph $F'$ with chromatic number $k:=\chi(F)-1$ and a color-critical edge. By a result of Simonovits \cite{sim}, we have that $\ex(s,\cF)=|E(T(s,k-1))|$ if $s$ is large enough. Therefore, the above theorem indeed
generalizes the second result of Alon and Frankl \cite{af}. We also have the following.

\begin{corollary}
    If $\chi(F)>2$, then $\ex(n,\{F,M_{s+1}\})=s(n-s)+O(1)$.
\end{corollary}

In the case $F$ is bipartite, we can also determine $\ex(n,\{F,M_{s+1}\})$ apart from an additive constant term. Let $F$ be a bipartite graph and let $p=p(F)$ denote the smallest possible order of a color class in a proper two-coloring of $F$. If $p>s$, then $G(n,s)$ and $K_{2s+1}$ are both $F$-free, thus the Erd\H os-Gallai theorem \cite{eg} gives the exact value of $\ex(n,\{F,M_{s+1}\})$.

\begin{prop}\label{prop2} If $F$ is bipartite and $p=p(F)\le s$, then
$\ex(n,\{F,M_{s+1}\})=(p-1)n+O(1)$. Moreover, there is a $K=K(F,s)$ such that for any $n$, there is an $n$-vertex $\{F,M_{s+1}\}$-free graph with $|E(G)|=\ex(n,\{F,M_{s+1}\})$ that has vertices $v_1,\dots,v_{p-1}$ and at least $n-K$ vertices $u$ such that the neighborhood of $u$ is $\{v_1,\dots,v_{p-1}\}$. Furthermore, the vertices with neighborhood different from $\{v_1,\dots,v_{p-1}\}$ each have degree at least $p$.
\end{prop}

The lower bound is given by $K_{p-1,n-p+1}$. It is clearly not the extremal graph though. Now we describe two candidates. 

\bigskip

\noindent\textbf{Construction 1.} Let $\cF_0$ denote the family of graphs obtained by deleting $p-1$ vertices from $F$ and let $\cF_1=\cF_0\cup \{M_{s-p+2}\}$. Then we can add an $\cF_1$-free graph to the larger class of $K_{p-1,n-p+1}$ and all edges to the smaller class. The resulting graph is clearly $\{F,M_{s+1}\}$-free and has $(p-1)(n-p+1)+\binom{p-1}{2}+\ex(n-p+1,\cF_1)$ edges. Note that $\cF_1$ contains $K_{1,|V(F)|-p}$, thus $\ex(n-p+1,\cF_1)=O(1)$.

\smallskip

\noindent \textbf{Construction 2.} Assume that $F$ is connected. We take $K_{p-1,n+p-2s}$, and on the remaining $2s-2p+1$ vertices, we take an $F$-free graph with $\ex(2s-1,F)$ edges. Clearly, none of the components of this graph contains $F$, and the largest matchings have size at most $p-1+s-p$.

\bigskip

We remark that the second construction can easily be improved for some specific $F$. For example, if $F$ is a path $P_{4}$ on $4$ vertices, we can take $K_{p-1,n-3s+2p-1}$ and $s-p$ triangles. 
We claim that if $F$ contains a cycle and $s$ is large enough, then the second construction contains more edges. Indeed, compared to the first construction, we lose $O(s)$ edges and gain $\omega(s)$ edges.

Assume now that $F$ is a forest and observe that $\cF_1$ contains a matching of order at most $|V(F)|-p+1$. Indeed, if $F$ has $v$ non-isolated vertices, then there are at most $v-1$ edges between the two parts, thus at most $p-1$ vertices of the part of order $|V(F)|-p$ have degree more than 1. If we delete those vertices, we obtain a matching. This implies that $\ex(n-p+1,\cF_1)$ does not depend on $s$. 

Now assume that $F$ is a tree with parts of different order, i.e., $|V(F)|>2p$. Assume furthermore that $s$ and $n$ are sufficiently large, and for simplicity assume that $2s-1$ is divisible by $|V(F)|-1$. In this case $s/(|V(F)|-1)$ copies of $K_{|V(F)|-1}$ forms an $F$-free graph, thus $\ex(2s-1,F)\ge (|V(F)|-2)(2s-1)/2$. Now, compared to Construction 1, the second construction loses $(2s-1)(p-1)+c$ edges, where $c$ does not depend on $s$. On the other hand, Construction 2 gains at least $(2s-1)(|V(F)|-2)/2>(2s-1)(p-1/2)$, thus Construction 2 is better. Note that essentially the same argument also works if $2s-1$ is not divisible by $|V(F)|-1$.

We believe that for other trees Construction 1 is better than Construction 2 for every $s$, moreover, Construction 1 is extremal. The Erd\H os-S\'os conjecture \cite{erd} states that for any tree $F$, we have $\ex(n,F)\le (|V(F)|-2)n/2$. It is known for several classes of trees. In particular, it was shown for paths by Erd\H os and Gallai \cite{eg}.

\begin{prop}\label{prop3}
Let $F$ be a balanced tree, i.e., $|V(F)|=2p(F)$ and let $p(F)\le s$. Assume that the Erd\H os-S\'os conjecture holds for $F$. Then for sufficiently large $n$, we have $\ex(n,\{F,M_{s+1}\})=(p-1)(n-p+1)+\binom{p-1}{2}$.    
\end{prop}

The above proposition determines $\ex(n,\{P_{2\ell},M_{s+1}\})$ for sufficiently large $n$. We can also deal with odd paths.

\begin{prop}\label{prop4} Let $2\le\ell\le s$.
    If $\ell$ divides $s-\ell+1$, then for sufficiently large $n$ we have that $\ex(n,\{P_{2\ell+1},M_{s+1}\})=(\ell-1)(n-2s+\ell-1)+\binom{\ell-1}{2}+(s-\ell+1)(2\ell-1)$. If $\ell$ does not divide $s-\ell+1$, then let $t:=\lfloor(s-\ell+1)/\ell\rfloor$. For sufficiently large $n$, we have that $\ex(n,\{P_{2\ell+1},M_{s+1}\})=(\ell-1)(n-\ell+1-2\ell t)+1+\binom{\ell-1}{2}+t\binom{2\ell}{2}$.
\end{prop}


\section{Proofs}

Let us start with the proof of Theorem \ref{thm1} that we restate here for convenience.

\begin{thm*}
If $\chi(F)>2$ and $n$ is large enough, then $\ex(n,\{F,M_{s+1}\})=\ex(s,\cF)+s(n-s)$, where $\cF$ is the family of graphs obtained by deleting an independent set from $F$.
\end{thm*}

\begin{proof}
Let $G_0$ be an $s$-vertex $\cF$-free graph with $\ex(s,\cF)$ edges. Let us add $n-s$ new vertices and connect each of them to each vertex of $G_0$. The resulting graph is clearly $M_{s+1}$-free, since $s$ vertices are incident to all the edges, and $F$-free by the definition of $\cF$. This gives the lower bound.


To show the upper bound, consider an $\{F,M_{s+1}\}$-free $n$-vertex graph $G$. Let $v_1,\dots,v_n$ be the vertices of $G$ in decreasing order of their degrees. Observe that $d(v_{s+1})\le 2s$. Indeed, otherwise we can pick greedily a matching $M_{s+1}$ the following way. In step $i$, we pick $v_i$ and a neighbor of $v_i$ we have not picked earlier. This way we have at most $2i-2$ forbidden neighbors, thus we can pick a new one even at step $s+1$, a contradiction. 

Observe also that $G$ has at most $\sum_{i=1}^{2s}d(v_i)\le \sum_{i=1}^{s}d(v_i)+2s^2$ edges. Indeed, the at most $2s$ vertices of a largest matching are incident to every edge, and $2s$ vertices are incident to at most $\sum_{i=1}^{2s}d(v_i)$ edges. The upper bound on this quantity follows from $d(v_{s+1}),\dots,d(v_{2s})\le 2s$.

We claim that $d(v_s)\ge n-3s^2$. Indeed, otherwise $\sum_{i=1}^{s}d(v_i)+2s^2\le (s-1)(n-1)+n-s^2\le s(n-s)$ and we are done. This implies that $v_1,\dots,v_s$ have at least $n-s-3s^3$ common neighbors. Let $U=\{v_1,\dots,v_s\}$. Observe that $G[U]$ is $\cF$-free, otherwise we would find an $F$ by picking at most $|V(F)|$ of their common neighbors as the missing independent set. 

We claim that there is no edge outside $U$. Indeed, otherwise we could find $M_{s+1}$ greedily as earlier: first we pick the edge outside $U$, and then in step $i+1$, we pick $v_i$ and a neighbor of $v_i$ we have not picked earlier. This is doable since $v_i$ has at least $n-3s^2\ge 2i$ neighbors. The number of edges is at most $\ex(s,\cF)+s(n-s)$, where the first term is an upper bound on the number of edges inside $U$, while the second term is an upper bound on the number of edges with one endpoint inside $U$ and the other endpoint outside $U$. This completes the proof.
\end{proof}

Let us continue with the proof of Proposition \ref{prop2} that we restate here for convenience.

\begin{prop*} If $F$ is bipartite and $p=p(F)\le s$, then
$\ex(n,\{F,M_{s+1}\})=(p-1)n+O(1)$. Moreover, there is a $K=K(F,s)$ such that for any $n$, there is an $n$-vertex $\{F,M_{s+1}\}$-free graph with $|E(G)|=\ex(n,\{F,M_{s+1}\})$ that has vertices $v_1,\dots,v_{p-1}$ and at least $n-K$ vertices $u$ such that the neighborhood of $u$ is $\{v_1,\dots,v_{p-1}\}$. Furthermore, the vertices with neighborhood different from $\{v_1,\dots,v_{p-1}\}$ each have degree at least $p$.
\end{prop*}

\begin{proof}
The lower bound is given by $K_{p-1,n-p+1}$, or by Construction 1 or Construction 2.

Let $G$ be an $n$-vertex $\{F,M_{s+1}\}$-free graph. 
Let $U$ denote the set of at most $2s$ vertices of a largest matching, then every edge of $G$ is incident to at least one vertex of $U$. Every $p$-set in $U$ has less than $q:=|V(F)|-p$ common neighbors. As there are at most $\binom{2s}{p}$ $p$-sets in $U$, there are at most $\binom{2s}{p}(q-1)$ vertices outside $U$ that are adjacent to at least $p$ sets. 

Let $W$ denote the set of the other at least $n-\binom{2s}{p}(|V(F)|-p)-2s$ vertices outside $U$. Then vertices of $W$ have degree at most $p-1$. 
Note that by choosing $K$ sufficiently large, we can assume that $n$ is sufficiently large. In particular, if at most $\binom{2s}{p-1}\max\{|V(F)|,2s\}$ vertices in $W$ with degree $p-1$, then the number of edges is at most  $(p-2)n+O(1)$ and we are done. Otherwise, at least $\max\{|V(F)|,2s\}$ vertices of $W$ have the same $p-1$ neighbors $v_1,\dots,v_{p-1}$.

For any other vertex of $W$, we change its neighborhood to $v_1,\dots,v_{p-1}$ to obtain $G'$. If $G'$ contained $F$ or $M_{s+1}$, that would contain some of the vertices whose neighborhood was changed. But they could be replaced by vertices with the same neighborhood already in $G$, to obtain $F$ or $M_{s+1}$ in $G$. Therefore, $G'$ is $\{F,M_{s+1}\}$-free. Clearly $|E(G')|\ge |E(G)|$, hence if $G$ has $\ex(n,\{F,M_{s+1}\})$ edges, then so does $G'$. It is easy to see that $G'$ has $(p-1)n+O(1)$ edges and the desired additional property.
\end{proof}
Let us continue with the proof of Proposition \ref{prop3} that we restate here for convenience.
\begin{prop*}
Let $F$ be a balanced tree, i.e., $|V(F)|=2p(F)$ and let $p(F)\le s$. Assume that the Erd\H os-S\'os conjecture holds for $F$. Then for sufficiently large $n$, we have $\ex(n,\{F,M_{s+1}\})=(p-1)(n-p+1)+\binom{p-1}{2}$.    
\end{prop*}

\begin{proof}
    The lower bound is given by Construction 1, which is $G(n,p-1)$ in this case. Indeed, if we delete $p-1$ vertices in one of the parts of $F$ and leave only a leaf, then the resulting graph is a single edge and some isolated vertices. As $\cF_1$ contains this graph, $\ex(n-p+1,\cF_1)=0$.

    For the upper bound, let $G$ be a graph ensured by Proposition \ref{prop2}. Thus, $G$ has $n$ vertices, $\ex(n,\{F,M_{s+1}\})$ edges, $G$ is $\{F,M_{s+1}\}$-free, and $G$ contains a set $U=\{v_1,\dots,v_{p-1}\}$ such that all but $K$ vertices have neighborhood $U$. Let $W$ denote the set of vertices with neighborhood $U$ and $U':=V(G)\setminus (U\cup W)$. 
    There is no edge inside $W$ by definition. 
    
    \begin{clm}
    There is no edge between $U$ and $U'$.
    \end{clm}

    \begin{proof}
        First we show that if $F\neq K_2$, then $F$ has a vertex $x$ that is adjacent to at least one, but at most $p-1$ leaves and exactly one neighbor of degree greater than 1. Indeed, let $F'$ be the graph we obtain by deleting the leaves of $F$, then $F'$ has at least two leaves. Those vertices in $F$ have one neighbor of degree greater than 1 and at least 1 leaf neighbor. As there are at most $2p-2$ leaves in $F$, at least one of these two vertices have at most $p-1$ leaf neighbors.

        Assume that $v_iu$ is an edge between $U$ and $U'$ and let $u'$ be a neighbor of $u$ outside $U$ (this exists otherwise $u\in W$).
        Now we map $x$ to $u$ its non-leaf neighbor to $v_i$, and we map the leaf neighbors of $x$ to $u'$ and $p-2$ other neighbors of $u$. We map the remaining vertices of the part of $F$ containing these leaves to arbitrary vertices in $U$, and the remaining vertices of the other part of $F$ to arbitrary vertices in $W$. This way we find a copy of $F$ in $G$, a contradiction.
    \end{proof}
    
    

    
   Let us return to the proof of the proposition. Since the Erd\H os-S\'os conjecture holds for $F$, we have $\ex(|U'|,F)\le (p-1)|U'|$, thus there are at most $(p-1)|U'|$ edges inside $U'$. Then $|E(G)|\le \binom{p-1}{2}+(p-1)(n-p+1-|U'|)+\ex(|U'|,F)\le \binom{p-1}{2}+(p-1)(n-p+1)$, completing the proof. 
\end{proof}

We finish the paper with the proof of Proposition \ref{prop4} that we restate here for convenience.

\begin{prop*} Let $2\le\ell\le s$.
    If $\ell$ divides $s-\ell+1$, then for sufficiently large $n$ we have that $\ex(n,\{P_{2\ell+1},M_{s+1}\})=(\ell-1)(n-2s+\ell-1)+\binom{\ell-1}{2}+(s-\ell+1)(2\ell-1)$. If $\ell$ does not divide $s-\ell+1$, then let $t:=\lfloor(s-\ell+1)/\ell\rfloor$. For sufficiently large $n$, we have that $\ex(n,\{P_{2\ell+1},M_{s+1}\})=(\ell-1)(n-\ell+1-2\ell t)+1+\binom{\ell-1}{2}+t\binom{2\ell}{2}$.
\end{prop*}

\begin{proof}
    The lower bounds are given by the following graphs. If $\ell$ divides $s-\ell+1$, then we take $G(n-2s+2\ell-2,\ell-1)$, and on the remaining $2s-2\ell+2$ vertices, we take $(s-\ell+1)/\ell$ copies of $K_{2\ell}$. Each component is $P_{2\ell+1}$-free, and the largest matching is of size $\ell-1$ in the large component, and of size $s-\ell+1$ in the clique components.

    If $\ell$ does not divide $s-\ell+1$, then we similarly take copies of $K_{2\ell}$ on at most $2s-2\ell+1$ vertices, i.e., we take $t$ copies. On the remaining $n-2\ell t$ vertices, we take $G(n-2\ell t,\ell-1)$ and add another edge. Again each component is $P_{2\ell+1}$-free, but this time the largest matching is of size $\ell$ in the large component. However, the remaining components have order $t2\ell<2s-2\ell+2$, thus the largest matching in those components have size at most $s-\ell$.

    Let us continue with the upper bounds. We apply Proposition \ref{prop2} to obtain an extremal $n$-vertex graph $G$ with vertices $U=\{v_1,\dots,v_{\ell-1}\}$, such that the set $W$ of vertices with neighborhood $U$ contains all but at most $K$ vertices. Moreover, the vertices of $U'=V(G)\setminus (U\cup W)$ have degree at least $\ell$. We will use multiple times the following simple observation: changing the neighborhood of a vertex $u$ to $U$ does not create $F$ or $M_{s+1}$. Indeed, we could replace the vertex $u$ in any forbidden configuration to any other common neighbor of the vertices of $U$ to create another copy without containing any of the new edges.

    There is no edge inside $W$ by definition. We claim that if there is an edge $uv$ with $u\in U$ and $v\in U'$, then the component $C$ of $v$ in $G[U']$ is a single edge. Indeed, $v$ has at least $\ell$ neighbors, thus a neighbor $w$ outside $U$, which must be in $U'$. If $w$ has another neighbor $w'$ in $U'$, then $u_1v_1u_2\dots u_{\ell-1}v_{\ell-1}vww'$ is a $P_{2\ell+1}$, where $u_i$ are arbitrary distinct elements of $W$ and we assumed $u=v_{\ell-1}$ without loss of generality. This implies that $C$ is a star with center $v$. But if $w$ has no other neighbor in $U'$, then it has a neighbor in $U$ (in fact $\ell-1$ neighbors), hence the component of $w$ in $G[U']$ (which is $C$) must be a star with center $w$.

    We also claim that there is at most one such edge component. Indeed, its vertices are joined to each vertex of $U$, thus two such edges $vw$ and $v'w'$ create a $P_{2\ell+1}$ of the form $v'w'v_1u_2\dots u_{\ell-1}v_{\ell-1}vw$ (where $u_i$ are arbitrary distinct elements of $W$).


    Consider a component $C$ of $U'$ that is not a single edge. If $C$ does not contain $P_{2\ell}$, then it contains at most $\ex(|V(C)|,P_{2\ell})=|V(C)|(\ell-1)$ edges. Then we can change the neighborhood of vertices in $C$ to $U$. The resulting graph is also $\{P_{2\ell+1},M_{s+1}\}$-free and the number of edges does not decrease. We apply these to all the $P_{2\ell}$-free components. In the resulting graph $G'$, every vertex of $U'$ is in a component containing a $P_{2\ell}$, in particular is the endvertex of a $P_{\ell+1}$ inside $U'$. As every vertex of $U$ is the endvertex of a $P_{2\ell-1}$ outside $U'$, an edge between $U$ and $U'$ would give a $P_{2\ell+1}$ in $G'$, a contradiction.

    Consider now a component $C$ of $G$ in $U'$ with $v>2\ell$ vertices. A theorem of Kopylov \cite{kopy} gives an upper bound on the number of edges inside $P_k$-free connected graphs. It shows that $|E(G[C])|\le \max\{\binom{2\ell-1}{2}+v-2\ell+1,|E(G(v,\ell-1))|+1\}\le v(\ell-1)$. Therefore, again, we can change the neighborhood of vertices in $C$ to $U$ without decreasing the number of edges.

    Consider now a component $C$ of $G$ in $U'$ with less than $2\ell$ vertices. Then $C$ has at most $|V(C)|(\ell-1)$ edges, thus again, we can change the neighborhood of vertices in $C$ to $U$ without decreasing the number of edges.

    Consider now a component $C$ of $G$ in $U'$ with $2\ell$ vertices that is $M_\ell$-free. By the Erd\H os-Gallai theorem, we know that $C$ contains at most $\binom{2\ell}{2}-\ell+1\le 2\ell(\ell-1)$ edges, thus again, we can change the neighborhood of vertices in $C$ to $U$ without decreasing the number of edges.

    We obtained that each component in $G[U']$ (except at most one component of order 2) has $2\ell$ vertices and contains a matching $M_\ell$, thus adding the missing edges inside that component would not increase the largest matching in $G$, nor it would create $P_{2\ell+1}$. Therefore, $U'$ consists of copies of $K_{2\ell}$. Clearly there are at most $t$ copies. Clearly, $2\ell$ vertices in a $K_{2\ell}$ add $\ell(2\ell-1)$ edges, while $2\ell$ vertices in $W$ add $2\ell(\ell-1)$ edges. Therefore, it is worth to pick the maximum number of $2\ell$-cliques. 

    If there is no component of order 2 in $G[U']$
    or $\ell$ does not divide $s-\ell+1$, then we are done. In the remaining case, we can only add $t-1$ copies of $K_{2\ell}$. Compared to this graph, we can delete $2\ell$ vertices from $W$ including the endvertices of the extra edge from $G$ and add one more copy of $K_{2\ell}$. This way we removed $2\ell(\ell-1)+1$ edges and added $\ell(2\ell-1)$ edges without creating $F$ or $M_{s+1}$. The number of edges increases, a contradiction completing the proof.
\end{proof}

\bigskip

\textbf{Funding}: Research supported by the National Research, Development and Innovation Office - NKFIH under the grants SNN 129364, FK 132060, and KKP-133819.

\end{document}